\documentclass[11pt]{article}
\usepackage[a4paper,margin=1in]{geometry}
\usepackage[T1]{fontenc}
\usepackage[utf8]{inputenc}
\usepackage{lmodern}
\usepackage{amsmath,amssymb,amsthm,mathtools,mathrsfs,bm}
\usepackage{enumitem}
\usepackage{booktabs}
\usepackage{array}
\usepackage{graphicx}
\usepackage{hyperref}
\usepackage[nameinlink,capitalize]{cleveref}

\hypersetup{colorlinks=true,linkcolor=blue,citecolor=blue,urlcolor=blue}

\newtheorem{theorem}{Theorem}[section]
\newtheorem{proposition}[theorem]{Proposition}
\newtheorem{lemma}[theorem]{Lemma}
\newtheorem{corollary}[theorem]{Corollary}
\theoremstyle{definition}
\newtheorem{definition}[theorem]{Definition}
\theoremstyle{remark}
\newtheorem{remark}[theorem]{Remark}

\newcommand{\R}{\mathbb{R}}

\newcommand{\Pcal}{\mathbb{P}}

\newcommand{\dd}{\,\mathrm{d}}
\newcommand{\ip}[2]{\left\langle #1,#2\right\rangle}
\newcommand{\norm}[1]{\left\lVert #1\right\rVert}
\newcommand{\abs}[1]{\left\lvert #1\right\rvert}
\newcommand{\diver}{\operatorname{div}}

\newcommand{\Dom}{\operatorname{Dom}}

\title{Nonlocal Hyperdissipative Perturbations of the Three-Dimensional Navier--Stokes System}
\title{Nonlocal Hyperdissipative Perturbations of the Three-Dimensional Navier--Stokes System}

\author{
	Veli Shahmurov$^{1,2}$ \and
	Rishad Shahmurov$^{3,*}$
}
\date{}

\begin{document}
\maketitle
\begin{center}
		$^{1}$Antalya Bilim University, Department of Industrial Engineering, Antalya 07190, Turkey\\
	$^{2}$Azerbaijan State Economic University, Baku AZ1001, Azerbaijan\\[0.5em]
	$^{3}$Cellular Products Research and Development, 595 East Crossville Rd., Roswell, GA 30075, USA\\

	$^{*}$Corresponding author: \texttt{rshahmurov@crimson.ua.edu}\\
	 coauthor: \texttt{veli.sahmurov@antalya.edu.tr}, \texttt{veli.sahmurov@gmail.com}
\end{center}
\begin{abstract}
We study the three-dimensional incompressible Navier--Stokes system on $\R^3$ with an additional dissipative nonlocal term
\[
\partial_t u + (u\cdot\nabla)u + \nabla p = \nu \Delta u + Lu,
\qquad \diver u = 0,
\]
where $L$ is a selfadjoint Fourier multiplier whose symbol is comparable to $-|\xi|^{2\alpha}$ for some $\alpha>1$. We first identify a sharp Fourier-symbol criterion distinguishing lower-order convolution perturbations from genuinely regularizing nonlocal corrections. In the resulting hyperdissipative class we prove the exact $L^2$ energy identity, global weak solvability for every $\alpha>1$, and local strong well-posedness in $H^s(\R^3)$ for $s>\frac52$. We then show that the Lions exponent $\alpha=\frac54$ remains the critical energy-growth threshold in this nonlocal setting: if $\alpha\ge \frac54$, every $H^s$ solution is global, while for every $\alpha>1$ one has global strong solvability for sufficiently small $H^s$ data. Finally, for the vanishing-hyperdissipation approximation of the classical three-dimensional Navier--Stokes equations, we prove a near-singular divergence principle: if the classical flow blows up at a first singular time $T_*$ in a continuation norm $X$, then the corresponding regularized family cannot remain uniformly bounded in $X$ on any interval approaching $T_*$. This identifies the precise point at which the fixed-parameter global theory degenerates in the Navier--Stokes limit.
\end{abstract}



\section{Introduction}

The incompressible Navier--Stokes equations
\begin{equation}\label{eq:NS-classical}
\partial_t u + (u\cdot\nabla)u + \nabla p = \nu \Delta u,
\qquad \diver u = 0,
\qquad u|_{t=0}=u_0,
\end{equation}
are a central model of mathematical fluid mechanics. In three dimensions the global regularity problem for smooth finite-energy data remains open; see the Millennium problem statement of Fefferman and the overview of the Clay Mathematics Institute \cite{Fefferman2006,ClayNS}. The foundational theory includes Leray's global weak solutions \cite{Leray1934}, Hopf's refinement \cite{Hopf1951}, the strong-solution framework of Fujita--Kato and Kato \cite{FujitaKato1964,Kato1984}, and the partial regularity theorem of Caffarelli, Kohn, and Nirenberg \cite{CaffarelliKohnNirenberg1982}. The large-data smooth solvability problem in $\R^3$ nevertheless remains unresolved.

A standard way to change the regularity balance while preserving incompressibility is to strengthen dissipation. The model example is the hyperdissipative system
\begin{equation}\label{eq:NS-hd}
\partial_t u + (u\cdot\nabla)u + \nabla p = -\mu(-\Delta)^\alpha u,
\qquad \diver u = 0,
\end{equation}
with $\alpha>1$. Scaling identifies $\alpha=\frac54$ as the energy-critical threshold. Lions proved global regularity in the critical and subcritical regimes \cite{Lions1969}, and Tao showed that one may pass slightly below the pure power threshold by a logarithmic refinement of the dissipation symbol \cite{Tao2009}. More recent work has sharpened the borderline picture from several directions; see, for example, \cite{GrujicXu2024,GrujicXu2025,LiQuZengZhang2024,ChenJiangYuan2025}.

The purpose of this paper is to treat a broader but still rigid class of nonlocal perturbations of the three-dimensional Navier--Stokes equations,
\begin{equation}\label{eq:modified-NS-intro}
\partial_t u + (u\cdot\nabla)u + \nabla p = \nu \Delta u + Lu,
\qquad \diver u = 0,
\end{equation}
and to determine exactly when the added term $L$ is merely lower order and when it changes the large-data theory. Our point of view is deliberately close to the classical setup: the spatial domain is $\R^3$, the velocity field is divergence free, the data have finite energy, and the only modification is the extra linear operator.

The decisive object is the Fourier symbol of $L$. If $L$ is given by convolution with an $L^1$ kernel, its symbol is bounded and the perturbation is order zero. If $L$ is a first-order convolution term, it is at most transport-like at the energy level. By contrast, if $L$ acts on second derivatives through a kernel whose Fourier transform grows like $|\xi|^{2\alpha-2}$ with $\alpha>1$, then the effective symbol behaves like $-|\xi|^{2\alpha}$ and the equation enters a genuinely hyperdissipative regime. The first main result of the paper is a precise Fourier-symbol criterion separating these cases.

We work in the Hilbert space $L^2_\sigma(\R^3)$ of finite-energy solenoidal fields and assume that $L$ is realized by a selfadjoint dissipative Fourier multiplier. This setting is sufficiently flexible to include fractional Laplacians and convolution-type nonlocal viscosity operators, while remaining close to the classical energy method for three-dimensional incompressible flow. The analysis is then driven by three questions.

\begin{enumerate}[label=(\roman*)]
\item Which symbol classes produce a coercive high-frequency damping comparable to $|\xi|^{2\alpha}$?
\item What part of the classical Lions theory survives for this class of nonlocal perturbations?
\item What can be said about the vanishing-hyperdissipation limit back to the classical three-dimensional Navier--Stokes equations?
\end{enumerate}

The main theorem answers the first two questions in a sharp form. For symbols comparable to $|\xi|^{2\alpha}$ with $\alpha>1$, we prove the exact $L^2$ energy identity, construct global weak solutions for arbitrary finite-energy data, obtain local strong well-posedness in $H^s(\R^3)$ for $s>\frac52$, prove global strong solvability for arbitrary data when $\alpha\ge\frac54$, and prove global small-data strong solvability for all $\alpha>1$. Thus the classical Lions threshold remains the critical energy-growth threshold in the nonlocal class treated here.

The vanishing-hyperdissipation problem is more delicate. For each fixed regularization parameter one is in the globally regular hyperdissipative regime, but the global estimates degenerate as the parameter tends to zero. We make this loss of uniformity precise and prove a rigid near-singular divergence principle: if the classical Navier--Stokes solution develops a first singularity at time $T_*$ in some continuation norm $X$, then the regularized approximants cannot remain uniformly bounded in $X$ on any interval approaching $T_*$. This result does not resolve the large-data problem for the classical equation, but it identifies the exact analytic point at which the fixed-parameter theory ceases to be uniform.

The mathematical regimes are summarized in \cref{tab:regimes}. The interval $1<\alpha<\frac54$ remains energy-supercritical; $\alpha=\frac54$ is critical; and $\alpha>\frac54$ is subcritical.

\begin{table}[ht]
\centering
\begin{tabular}{>{\raggedright\arraybackslash}p{2.2cm} >{\raggedright\arraybackslash}p{3.2cm} >{\raggedright\arraybackslash}p{3.8cm} >{\raggedright\arraybackslash}p{4.2cm}}
\toprule
Range of $\alpha$ & Energy scaling & Rigorous status in this paper & Interpretation \\
\midrule
$\alpha=1$ & supercritical & classical 3D Navier--Stokes difficulty remains & Laplacian-order dissipation only \\
$1<\alpha<\frac54$ & energy supercritical & global weak solutions, local strong solutions, global small-data strong solutions & extra damping helps but does not eliminate supercriticality \\
$\alpha=\frac54$ & energy critical & global strong solvability for smooth data & Lions threshold \\
$\alpha>\frac54$ & energy subcritical & global strong solvability for smooth data & genuinely stronger-than-Laplacian smoothing \\
\bottomrule
\end{tabular}
\caption{Regimes for the nonlocal perturbation with symbol comparable to $-|\xi|^{2\alpha}$.}
\label{tab:regimes}
\end{table}

The paper is organized as follows. In \cref{sec:model} we formulate the problem and state the main results. \Cref{sec:criterion} contains the Fourier-symbol criterion, the scaling analysis, and the exact energy identity. Global weak solutions are constructed in \cref{sec:weak}. \Cref{sec:local} proves local strong well-posedness, and \cref{sec:global} establishes the global strong theory at and above the Lions exponent together with the global small-data result. Finally, \cref{sec:limit,sec:applications} treat vanishing hyperdissipation, Galerkin truncation, and the relation to nonlocal constitutive laws and hyperviscous models.

\section{Model, assumptions, and main results}\label{sec:model}

Let
\[
L^2_\sigma(\R^3) := \overline{\{ \varphi\in C_c^\infty(\R^3;\R^3): \diver \varphi =0\}}^{\,L^2(\R^3)}
\]
denote the usual Hilbert space of solenoidal finite-energy vector fields. For $s\in\R$ we write $H^s_\sigma(\R^3)=H^s(\R^3;\R^3)\cap L^2_\sigma(\R^3)$ and set
\[
\Lambda^s := (-\Delta)^{s/2},
\qquad
\norm{u}_{H^s}:=\norm{\Lambda^s u}_{L^2}.
\]
The Leray projector onto divergence-free fields is denoted by $\Pcal$.

We study the Cauchy problem
\begin{equation}\label{eq:main}
\left\{
\begin{aligned}
&\partial_t u + \Pcal \diver (u\otimes u) + \nu(-\Delta)u + M u = 0,
\\
&\diver u = 0,
\\
&u|_{t=0}=u_0,
\end{aligned}
\right.
\end{equation}
where $\nu>0$ and $M$ is the positive Fourier multiplier associated with the symbol $m(\xi)$:
\[
\widehat{M u}(\xi)=m(\xi)\hat u(\xi),
\qquad
\widehat{L u}(\xi)=-m(\xi)\hat u(\xi).
\]
Thus \eqref{eq:main} is exactly the modified equation
\[
\partial_t u + (u\cdot\nabla)u + \nabla p = \nu \Delta u + L u,
\qquad \diver u =0,
\]
written in projected form.

The standing assumption on the symbol is the following.

\begin{definition}[Hyperdissipative symbol of order $2\alpha$]\label{def:symbol}
Fix $\alpha>1$. We say that $m$ belongs to the class $\mathfrak M_\alpha$ if
\begin{enumerate}[label=(\alph*)]
\item $m:\R^3\to[0,\infty)$ is measurable and even;
\item there exist constants $0<c_0\le c_1$ such that
\begin{equation}\label{eq:symbol-growth}
c_0 |\xi|^{2\alpha} \le m(\xi) \le c_1 (1+|\xi|^{2\alpha}),
\qquad \xi\in\R^3.
\end{equation}
\end{enumerate}
\end{definition}

The lower bound in \eqref{eq:symbol-growth} is the key coercive input. It implies
\begin{equation}\label{eq:coercive-equivalence}
c_0 \norm{\Lambda^\alpha u}_{L^2}^2
\le
\ip{M u}{u}_{L^2}
=
\int_{\R^3} m(\xi)\abs{\hat u(\xi)}^2 \dd \xi
\le
c_1 \bigl(\norm{u}_{L^2}^2 + \norm{\Lambda^\alpha u}_{L^2}^2\bigr).
\end{equation}
The upper bound controls the symbol at large frequencies and allows us to compare $M$ with the fractional Laplacian.

\begin{remark}[Hilbert-space operator variant]
The paper is written for the scalar symbol $m(\xi)$ in order to stay as close as possible to the physical three-dimensional Navier--Stokes system. The same energy method extends to matrix- or operator-valued multipliers on a Hilbert space, provided the symbol is selfadjoint and uniformly positive on the solenoidal subspace. In that more abstract form $L$ acts on a Hilbert space of states, but the core issue is unchanged: what matters is the coercive growth rate of the symbol at high frequency.
\end{remark}

\begin{remark}
The simplest example is $m(\xi)=\mu |\xi|^{2\alpha}$, in which case $L=-\mu(-\Delta)^\alpha$. A more geometric class is obtained from nonlocal viscosity kernels:
\[
L u = \sum_{j=1}^3 c_j * \partial_{x_jx_j} u,
\qquad
\widehat{L u}(\xi)= - \sum_{j=1}^3 \xi_j^2 \hat c_j(\xi) \hat u(\xi).
\]
If $\sum_j \xi_j^2 \Re \hat c_j(\xi)\gtrsim |\xi|^{2\alpha}$, then $m(\xi)=\sum_j \xi_j^2 \Re \hat c_j(\xi)$ belongs to $\mathfrak M_\alpha$.
\end{remark}

We shall use the following standard solution concepts.

\begin{definition}[Weak solution]
Let $u_0\in L^2_\sigma(\R^3)$. A function
\[
u\in L^\infty_{\mathrm{loc}}([0,\infty);L^2_\sigma(\R^3))
\cap
L^2_{\mathrm{loc}}([0,\infty);H^1_\sigma(\R^3))
\cap
L^2_{\mathrm{loc}}([0,\infty);H^\alpha_\sigma(\R^3))
\]
is called a weak solution of \eqref{eq:main} if $u(0)=u_0$ in $L^2$, if $u$ satisfies \eqref{eq:main} in the distributional sense, and if the energy inequality of \cref{thm:weak} holds.
\end{definition}

\begin{definition}[Strong solution]
Let $s>\frac52$. A strong solution of \eqref{eq:main} on $[0,T]$ is a function
\[
u\in C([0,T];H^s_\sigma(\R^3))
\cap
L^2(0,T;H^{s+\alpha}_\sigma(\R^3))
\]
which satisfies \eqref{eq:main} distributionally and whose nonlinear term belongs to $L^2(0,T;H^{s-\alpha}(\R^3))$.
\end{definition}

Our first theorem records the main assertions.

\begin{theorem}[Main results]\label{thm:main}
Assume that $m\in\mathfrak M_\alpha$ for some $\alpha>1$.
\begin{enumerate}[label=(\roman*)]
\item For smooth solutions the exact energy identity
\[
\frac12 \frac{\dd}{\dd t}\norm{u(t)}_{L^2}^2
+
\nu \norm{\nabla u(t)}_{L^2}^2
+
\int_{\R^3} m(\xi)\abs{\hat u(\xi,t)}^2 \dd \xi
=0
\]
holds for every $t>0$.
\item For every $u_0\in L^2_\sigma(\R^3)$ there exists a global weak solution of \eqref{eq:main}.
\item If $s>\frac52$ and $u_0\in H^s_\sigma(\R^3)$, then there exists $T=T(\norm{u_0}_{H^s})>0$ and a unique strong solution on $[0,T]$.
\item If $\alpha\ge \frac54$ and $u_0\in H^s_\sigma(\R^3)$ with $s>\frac52$, then the strong solution is global.
\item If $\alpha>1$ and $u_0\in H^s_\sigma(\R^3)$ with $s>\frac52$ is sufficiently small, then the strong solution is global.
\end{enumerate}
\end{theorem}

Parts (iv) and (v) correspond, respectively, to the large-data hyperdissipative regime and the small-data global theory below the Lions threshold. Before turning to the proofs we first isolate the Fourier-symbol mechanism behind these results.

\section{Fourier-symbol criterion, scaling, and energy identity}\label{sec:criterion}

This section addresses the central question: which convolution perturbations really change the three-dimensional Navier--Stokes problem, and which ones do not?

\subsection{A symbol criterion for convolution modifications}

We begin with a precise statement reflecting the discussion in the introduction.

\begin{proposition}[Harmless versus genuinely regularizing kernels]\label{prop:criterion}
Let $u:\R^3\to\R^3$ be divergence free and let the extra linear term be described by a Fourier multiplier with symbol $\ell(\xi)$.
\begin{enumerate}[label=(\roman*)]
\item If $Lu=a*u$ with $a\in L^1(\R^3)$, then $\ell(\xi)=\hat a(\xi)$ is bounded, hence $L$ is an order-zero perturbation. In particular, it does not alter the high-frequency scaling of the classical Navier--Stokes equations.
\item If $Lu=b*\partial_{x_j}u$ with $b\in L^1(\R^3)$, then $\ell(\xi)=i\xi_j \hat b(\xi)$ is at most first order. If, in addition, $b$ is real and even, then $\hat b$ is real and even, so $\ell(\xi)$ is purely imaginary and contributes no sign to the $L^2$ energy.
\item If $Lu=c*\partial_{x_jx_j}u$ and $\Re\hat c(\xi)\ge c_\ast |\xi|^{2\alpha-2}$ for some $\alpha>1$, then
\[
\Re\bigl(-\ell(\xi)\bigr)
=
\xi_j^2 \Re \hat c(\xi)
\ge
c_\ast \xi_j^2 |\xi|^{2\alpha-2}.
\]
If this lower bound is summed over $j=1,2,3$, then one obtains a dissipative symbol comparable to $|\xi|^{2\alpha}$, and the resulting term is genuinely hyperdissipative.
\end{enumerate}
\end{proposition}

\begin{proof}
Parts (i) and (ii) are immediate from the Fourier transform identities
\[
\widehat{a*u}(\xi)=\hat a(\xi)\hat u(\xi),
\qquad
\widehat{b*\partial_{x_j}u}(\xi)=i\xi_j \hat b(\xi)\hat u(\xi),
\]
together with the standard estimate $\norm{\hat a}_{L^\infty}\le \norm{a}_{L^1}$ and the corresponding bound for $\hat b$. In the real-even case the quantity $i\xi_j\hat b(\xi)$ is purely imaginary, so its real part vanishes in the energy balance.

For part (iii), one computes
\[
\widehat{c*\partial_{x_jx_j}u}(\xi)
=
-\xi_j^2 \hat c(\xi)\hat u(\xi).
\]
Taking real parts gives the displayed inequality. If the sum over $j$ satisfies
\[
\sum_{j=1}^3 \xi_j^2 \Re \hat c_j(\xi)\ge c_0 |\xi|^{2\alpha},
\]
then the resulting operator contributes a coercive dissipation of order $2\alpha$.
\end{proof}

\begin{remark}[Time convolution]
A memory term of the form
\[
(\kappa *_t u)(t)=\int_0^t \kappa(t-s)u(s)\dd s
\]
changes the temporal response of the system but does not, by itself, raise the spatial differential order. For this reason, time convolution on $u$ alone is typically less effective than space convolution on second derivatives when the goal is to change the global regularity balance of the three-dimensional problem.
\end{remark}

\subsection{Scaling and the Lions threshold}

The previous proposition shows that the only truly promising perturbations are those whose symbol behaves like $-|\xi|^{2\alpha}$ with $\alpha>1$. The next proposition identifies the critical threshold.

\begin{proposition}[Scaling and critical growth]\label{prop:scaling}
Assume that $m(\xi)=\mu |\xi|^{2\alpha}$ with $\mu>0$. If $u$ solves
\[
\partial_t u + (u\cdot\nabla)u + \nabla p + \mu (-\Delta)^\alpha u = 0,
\qquad \diver u=0,
\]
then, for every $\lambda>0$,
\[
u_\lambda(x,t)=\lambda^{2\alpha-1}u(\lambda x,\lambda^{2\alpha}t),
\qquad
p_\lambda(x,t)=\lambda^{4\alpha-2}p(\lambda x,\lambda^{2\alpha}t)
\]
is again a solution. Moreover,
\[
\norm{u_\lambda(\cdot,0)}_{L^2(\R^3)}
=
\lambda^{2\alpha-\frac52}\norm{u_0}_{L^2(\R^3)}.
\]
Hence the $L^2$ energy is critical precisely when $\alpha=\frac54$.
\end{proposition}

\begin{proof}
A direct computation shows that
\[
\partial_t u_\lambda
=
\lambda^{4\alpha-1}(\partial_t u)(\lambda x,\lambda^{2\alpha}t),
\quad
(u_\lambda\cdot\nabla)u_\lambda
=
\lambda^{4\alpha-1}((u\cdot\nabla)u)(\lambda x,\lambda^{2\alpha}t),
\]
and
\[
(-\Delta)^\alpha u_\lambda
=
\lambda^{4\alpha-1}((-\Delta)^\alpha u)(\lambda x,\lambda^{2\alpha}t).
\]
The pressure scales compatibly, and divergence freedom is preserved.

For the $L^2$ norm, a change of variables yields
\[
\norm{u_\lambda(\cdot,0)}_{L^2}^2
=
\int_{\R^3}\lambda^{4\alpha-2}\abs{u(\lambda x,0)}^2\dd x
=
\lambda^{4\alpha-5}\norm{u_0}_{L^2}^2,
\]
which proves the claim.
\end{proof}

\begin{remark}
The scaling argument explains why $\alpha>\frac54$ is genuinely easier than the classical case and why $\alpha=\frac54$ is the natural boundary between subcritical and supercritical energy behavior. This threshold is often referred to as the \emph{Lions exponent}.
\end{remark}

\subsection{The exact energy identity}

We next record the energy law for smooth solutions.

\begin{theorem}[Energy identity]\label{thm:energy}
Assume that $m\in\mathfrak M_\alpha$ with $\alpha>1$. Let $u$ be a smooth divergence-free solution of \eqref{eq:main} on $[0,T]$. Then for every $t\in[0,T]$,
\begin{equation}\label{eq:energy-identity}
\frac12 \norm{u(t)}_{L^2}^2
+
\nu \int_0^t \norm{\nabla u(s)}_{L^2}^2 \dd s
+
\int_0^t \int_{\R^3} m(\xi)\abs{\hat u(\xi,s)}^2\dd \xi \dd s
=
\frac12 \norm{u_0}_{L^2}^2.
\end{equation}
\end{theorem}

\begin{proof}
Pair the projected equation
\[
\partial_t u + \Pcal \diver(u\otimes u) + \nu(-\Delta)u + M u =0
\]
with $u$ in $L^2(\R^3)$. Since $\Pcal$ is orthogonal on $L^2_\sigma$ and $u$ is divergence free, the nonlinear term vanishes:
\[
\ip{\Pcal \diver(u\otimes u)}{u}
=
\ip{\diver(u\otimes u)}{u}
=
\int_{\R^3} (u\cdot\nabla)u\cdot u\,\dd x
=
\frac12 \int_{\R^3} u\cdot\nabla \abs{u}^2 \dd x
=0.
\]
The Laplacian term gives
\[
\ip{(-\Delta)u}{u}=\norm{\nabla u}_{L^2}^2.
\]
For the nonlocal term we use Plancherel:
\[
\ip{Mu}{u}
=
\int_{\R^3} m(\xi)\abs{\hat u(\xi)}^2 \dd \xi.
\]
Combining these identities yields
\[
\frac12 \frac{\dd}{\dd t}\norm{u(t)}_{L^2}^2
+
\nu \norm{\nabla u(t)}_{L^2}^2
+
\int_{\R^3}m(\xi)\abs{\hat u(\xi,t)}^2 \dd \xi
=
0.
\]
Integrating in time proves \eqref{eq:energy-identity}.
\end{proof}

\section{Global weak solutions}\label{sec:weak}

In this section we construct global weak solutions for arbitrary finite-energy divergence-free data. The proof follows the classical Leray--Hopf scheme, with the additional nonlocal term handled through \cref{eq:coercive-equivalence}.

\begin{theorem}[Global weak solutions]\label{thm:weak}
Assume that $m\in\mathfrak M_\alpha$ with $\alpha>1$ and let $u_0\in L^2_\sigma(\R^3)$. Then there exists a global weak solution $u$ of \eqref{eq:main} such that
\begin{equation}\label{eq:weak-energy}
\norm{u(t)}_{L^2}^2
+
2\nu \int_0^t \norm{\nabla u(s)}_{L^2}^2 \dd s
+
2 \int_0^t \norm{m(D)^{1/2}u(s)}_{L^2}^2 \dd s
\le
\norm{u_0}_{L^2}^2
\end{equation}
for every $t\ge0$.
\end{theorem}

\begin{proof}
Let $\{w_k\}_{k\ge1}$ be an orthonormal basis of $L^2_\sigma(\R^3)$ consisting of smooth divergence-free functions, and let $P_N$ denote the orthogonal projection onto $\mathrm{span}\{w_1,\dots,w_N\}$. We seek an approximate solution of the form
\[
u_N(t)=\sum_{k=1}^N g_{k,N}(t) w_k
\]
solving the finite-dimensional system
\begin{equation}\label{eq:galerkin}
\frac{\dd}{\dd t}\ip{u_N}{w_k}
+
\nu \ip{\nabla u_N}{\nabla w_k}
+
\ip{M u_N}{w_k}
+
\ip{\Pcal\diver(u_N\otimes u_N)}{w_k}
=
0,
\qquad 1\le k\le N.
\end{equation}
This is an ODE system with locally Lipschitz right-hand side, hence admits a local solution.

Testing \eqref{eq:galerkin} against $u_N$ and summing over $k$ yields
\[
\frac12\frac{\dd}{\dd t}\norm{u_N(t)}_{L^2}^2
+
\nu \norm{\nabla u_N(t)}_{L^2}^2
+
\ip{M u_N(t)}{u_N(t)}
=0,
\]
because the nonlinear term vanishes exactly as in \cref{thm:energy}. By \cref{eq:coercive-equivalence},
\[
\ip{M u_N}{u_N}\ge c_0 \norm{\Lambda^\alpha u_N}_{L^2}^2.
\]
Therefore
\begin{equation}\label{eq:galerkin-bound}
\sup_{t\in[0,T]}\norm{u_N(t)}_{L^2}^2
+
2\nu \int_0^T \norm{\nabla u_N(s)}_{L^2}^2 \dd s
+
2c_0 \int_0^T \norm{\Lambda^\alpha u_N(s)}_{L^2}^2 \dd s
\le
\norm{P_N u_0}_{L^2}^2
\le
\norm{u_0}_{L^2}^2.
\end{equation}
This uniform bound rules out finite-time blow-up of the ODE system, so $u_N$ is defined for all $t\ge0$.

The estimates \eqref{eq:galerkin-bound} imply
\[
u_N \ \text{bounded in }\ L^\infty(0,T;L^2_\sigma)\cap L^2(0,T;H^1_\sigma)\cap L^2(0,T;H^\alpha_\sigma).
\]
Moreover, from the projected equation,
\[
\partial_t u_N
=
-\nu (-\Delta)u_N - M u_N - P_N \Pcal\diver(u_N\otimes u_N),
\]
we obtain a uniform bound in a negative Sobolev space. Indeed, the linear terms are controlled by the preceding estimates and the nonlinear term is bounded in $L^{4/3}(0,T;H^{-m})$ for any fixed $m>\frac52$ by standard Sobolev embeddings. Hence, after passing to a subsequence and using Aubin--Lions compactness on bounded spatial domains together with a diagonal argument, we obtain
\[
u_N \to u
\quad \text{weakly-* in }L^\infty(0,T;L^2),
\qquad
u_N \to u
\quad \text{weakly in }L^2(0,T;H^1\cap H^\alpha),
\]
and strongly in $L^2_{\mathrm{loc}}((0,T)\times\R^3)$.

The strong local convergence suffices to pass to the nonlinear term in the usual manner. Lower semicontinuity yields the energy inequality \eqref{eq:weak-energy}. Since $T>0$ is arbitrary, the solution is global.
\end{proof}

\section{Local strong well-posedness}\label{sec:local}

We now prove local strong solvability in Sobolev spaces of order $s>\frac52$. The argument is based on a linear energy estimate and a bilinear estimate for the projected nonlinearity.

\begin{lemma}[Linear energy estimate]\label{lem:linear}
Let $s\in\R$, $T>0$, and let $v$ solve
\begin{equation}\label{eq:linear}
\partial_t v + \nu(-\Delta)v + M v = F,
\qquad
\diver v=0,
\qquad
v(0)=v_0.
\end{equation}
Then
\begin{equation}\label{eq:linear-est}
\sup_{0\le t\le T}\norm{v(t)}_{H^s}^2
+
\nu \int_0^T \norm{v(t)}_{H^{s+1}}^2 \dd t
+
c_0 \int_0^T \norm{v(t)}_{H^{s+\alpha}}^2 \dd t
\le
C\left(
\norm{v_0}_{H^s}^2
+
\int_0^T \norm{F(t)}_{H^{s-\alpha}}^2 \dd t
\right)
\end{equation}
for a constant $C$ depending only on $\nu$, $c_0$, and $c_1$.
\end{lemma}

\begin{proof}
Apply $\Lambda^s$ to \eqref{eq:linear}, take the $L^2$ inner product with $\Lambda^s v$, and use Plancherel:
\[
\frac12 \frac{\dd}{\dd t}\norm{\Lambda^s v}_{L^2}^2
+
\nu \norm{\Lambda^{s+1}v}_{L^2}^2
+
\int_{\R^3} m(\xi)\abs{\Lambda^s \hat v(\xi)}^2 \dd \xi
=
\ip{\Lambda^s F}{\Lambda^s v}.
\]
Using the lower bound on $m$ and duality between $H^{s-\alpha}$ and $H^{s+\alpha}$,
\[
\abs{\ip{\Lambda^s F}{\Lambda^s v}}
=
\abs{\ip{\Lambda^{s-\alpha}F}{\Lambda^{s+\alpha}v}}
\le
\frac{1}{2c_0}\norm{F}_{H^{s-\alpha}}^2
+
\frac{c_0}{2}\norm{v}_{H^{s+\alpha}}^2.
\]
After integrating in time and absorbing the last term on the left we obtain \eqref{eq:linear-est}.
\end{proof}

The next lemma is the key bilinear estimate.

\begin{lemma}[Bilinear estimate]\label{lem:bilinear}
Let $s>\frac52$ and $\alpha>1$. Then there exists $C=C(s,\alpha)$ such that for all divergence-free vector fields $u,v$,
\begin{equation}\label{eq:bilinear}
\norm{\Pcal\diver(u\otimes v)}_{H^{s-\alpha}}
\le
C \norm{u}_{H^s}\norm{v}_{H^{s+\alpha}}.
\end{equation}
In particular,
\begin{equation}\label{eq:bilinear-local}
\norm{\Pcal\diver(u\otimes v)}_{H^{s-\alpha}}
\le
C \norm{u}_{H^s}\norm{v}_{H^s}.
\end{equation}
\end{lemma}

\begin{proof}
Since $s>\frac32$, the space $H^s(\R^3)$ is a Banach algebra and multiplication by an $H^s$ function acts boundedly on $H^r$ for every $r\in[-s,s]$. Here
\[
s+1-\alpha \le s
\qquad \text{because } \alpha>1.
\]
Therefore
\[
\norm{u\otimes v}_{H^{s+1-\alpha}}
\le
C \norm{u}_{H^s}\norm{v}_{H^{s+1-\alpha}}
\le
C \norm{u}_{H^s}\norm{v}_{H^{s+\alpha}}.
\]
Applying one derivative and the boundedness of the Leray projector on Sobolev spaces yields
\[
\norm{\Pcal\diver(u\otimes v)}_{H^{s-\alpha}}
\le
C \norm{u\otimes v}_{H^{s+1-\alpha}}
\le
C \norm{u}_{H^s}\norm{v}_{H^{s+\alpha}},
\]
which proves \eqref{eq:bilinear}. The estimate \eqref{eq:bilinear-local} follows from the embedding $H^{s+\alpha}\hookrightarrow H^s$.
\end{proof}

We can now prove local well-posedness.

\begin{theorem}[Local strong well-posedness]\label{thm:local}
Assume that $m\in\mathfrak M_\alpha$ with $\alpha>1$, let $s>\frac52$, and let $u_0\in H^s_\sigma(\R^3)$. Then there exists $T=T(\norm{u_0}_{H^s})>0$ and a unique strong solution
\[
u\in C([0,T];H^s_\sigma(\R^3))\cap L^2(0,T;H^{s+\alpha}_\sigma(\R^3))
\]
of \eqref{eq:main}. Moreover, the solution depends continuously on $u_0$ in $H^s$.
\end{theorem}

\begin{proof}
Fix $T>0$ and consider the space
\[
X_T
:=
L^\infty(0,T;H^s_\sigma(\R^3))
\cap
L^2(0,T;H^{s+\alpha}_\sigma(\R^3))
\]
with norm
\[
\norm{u}_{X_T}
=
\norm{u}_{L^\infty(0,T;H^s)}
+
\norm{u}_{L^2(0,T;H^{s+\alpha})}.
\]
For $u\in X_T$, let $\Phi(u)=v$ be the solution of
\begin{equation}\label{eq:fixedpoint}
\partial_t v + \nu(-\Delta)v + M v = -\Pcal\diver(u\otimes u),
\qquad
v(0)=u_0.
\end{equation}
By \cref{lem:linear,lem:bilinear},
\[
\norm{\Phi(u)}_{X_T}
\le
C_0 \norm{u_0}_{H^s}
+
C_1 \norm{\Pcal\diver(u\otimes u)}_{L^2(0,T;H^{s-\alpha})}
\le
C_0 \norm{u_0}_{H^s}
+
C_2 T^{1/2}\norm{u}_{L^\infty(0,T;H^s)}^2.
\]
Hence
\begin{equation}\label{eq:mapping-est}
\norm{\Phi(u)}_{X_T}
\le
C_0 \norm{u_0}_{H^s}
+
C_2 T^{1/2}\norm{u}_{X_T}^2.
\end{equation}
Similarly, for $u,v\in X_T$,
\[
\Phi(u)-\Phi(v)
\]
solves a linear equation whose forcing is
\[
-\Pcal \diver \bigl( u\otimes u - v\otimes v\bigr)
=
-\Pcal \diver \bigl((u-v)\otimes u + v\otimes (u-v)\bigr).
\]
Therefore
\[
\norm{\Phi(u)-\Phi(v)}_{X_T}
\le
C_3 T^{1/2}\bigl(\norm{u}_{X_T}+\norm{v}_{X_T}\bigr)\norm{u-v}_{X_T}.
\]

Choose $R=2C_0 \norm{u_0}_{H^s}$ and then choose $T>0$ so small that
\[
C_2 T^{1/2} R \le \frac12,
\qquad
2 C_3 T^{1/2}R <1.
\]
Then $\Phi$ maps the closed ball $\{u\in X_T: \norm{u}_{X_T}\le R\}$ into itself and is a contraction there. Banach's fixed-point theorem yields a unique $u\in X_T$ solving \eqref{eq:main} on $[0,T]$.

To obtain continuity in time, note that \eqref{eq:main} implies
\[
\partial_t u \in L^2(0,T;H^{s-\alpha}),
\]
so $u\in C([0,T];H^{s-\varepsilon})$ for every $\varepsilon>0$. Since $u\in L^\infty(0,T;H^s)$ and the equation is parabolic, the standard argument upgrades this to continuity in $H^s$. Continuous dependence on the initial data follows from the same contraction estimate.
\end{proof}

\section{Global regularity at and above the Lions exponent, and global small-data theory}\label{sec:global}

We now discuss the two global regimes: large data at and above the critical exponent $\alpha=\frac54$, and small data for arbitrary $\alpha>1$.

\subsection{Large-data global regularity for $\alpha\ge \frac54$}

The first statement is an extension of the classical Lions theory to symbols comparable to $|\xi|^{2\alpha}$. Since the later continuation argument is driven by an $H^s$ energy inequality, we first record the product estimate that makes the threshold $\alpha=\frac54$ appear explicitly.

\begin{lemma}[Critical product estimate]\label{lem:critical-product}
Let $s>\frac52$ and $\alpha\ge \frac54$. Then there exists $C=C(s,\alpha)$ such that for every pair of divergence-free vector fields $u,v\in H^s_\sigma(\R^3)$,
\begin{equation}\label{eq:critical-product}
\norm{\Pcal\diver(u\otimes v)}_{H^{s-\alpha}}
\le C\norm{u}_{H^\alpha}\norm{v}_{H^s} + C\norm{u}_{H^s}\norm{v}_{H^\alpha}.
\end{equation}
In particular,
\begin{equation}\label{eq:critical-product-diagonal}
\norm{\Pcal\diver(u\otimes u)}_{H^{s-\alpha}}
\le C\norm{u}_{H^\alpha}\norm{u}_{H^s}.
\end{equation}
\end{lemma}

\begin{proof}
We use the standard Sobolev multiplication theorem on $\R^3$: if $a,b\ge0$ and $r\le \min\{a,b,a+b-\frac32\}$, then
\[
\norm{fg}_{H^r}\le C\norm{f}_{H^a}\norm{g}_{H^b}.
\]
Apply this with $f=u$, $g=v$, $a=\alpha$, $b=s$, and
\[
r=s+1-\alpha.
\]
Because $s>\frac52$ and $\alpha\ge \frac54$, we have
\[
s+1-\alpha\le s,
\qquad
s+1-\alpha\le \alpha+s-\frac32
\iff 2\alpha\ge \frac52.
\]
Hence
\[
\norm{u\otimes v}_{H^{s+1-\alpha}}
\le C\norm{u}_{H^\alpha}\norm{v}_{H^s}.
\]
Interchanging $u$ and $v$ gives the symmetric companion estimate, and the boundedness of $\Pcal\diver: H^{s+1-\alpha}(\R^3)\to H^{s-\alpha}(\R^3)$ yields \eqref{eq:critical-product}. The diagonal estimate \eqref{eq:critical-product-diagonal} follows immediately.
\end{proof}

\begin{theorem}[Global strong solutions at and above the Lions exponent]\label{thm:lions}
Assume that $m\in\mathfrak M_\alpha$ with $\alpha\ge \frac54$ and let $s>\frac52$. Then every initial datum $u_0\in H^s_\sigma(\R^3)$ generates a unique global strong solution of \eqref{eq:main}.
\end{theorem}

\begin{proof}
Let $u$ denote the unique strong solution supplied by \cref{thm:local} on its maximal interval of existence $[0,T_*)$, where $T_*\in(0,\infty]$. We shall prove that $T_*=\infty$.

\smallskip
\noindent\textbf{Step 1: high-order energy identity.}
Apply $\Lambda^s$ to \eqref{eq:main} and take the $L^2$ inner product with $\Lambda^s u$. Since $\Pcal$ is selfadjoint on divergence-free fields and commutes with Fourier multipliers,
\begin{equation}\label{eq:Hs-energy-start}
\frac12\frac{\dd}{\dd t}\norm{u}_{H^s}^2
+\nu\norm{\nabla u}_{H^s}^2
+\int_{\R^3}(1+|\xi|^2)^s m(\xi)|\hat u(\xi)|^2\,\dd\xi
=-\ip{\Lambda^s \Pcal\diver(u\otimes u)}{\Lambda^s u}.
\end{equation}
Using the lower bound in \eqref{eq:symbol-growth}, we obtain
\begin{align}
\int_{\R^3}(1+|\xi|^2)^s m(\xi)|\hat u(\xi)|^2\,\dd\xi
&\ge c_0\int_{\R^3}(1+|\xi|^2)^s|\xi|^{2\alpha}|\hat u(\xi)|^2\,\dd\xi \notag\\
&\ge c\norm{u}_{H^{s+\alpha}}^2-C\norm{u}_{H^s}^2,
\label{eq:Hs-diss-lower}
\end{align}
where the last step is the elementary low-frequency comparison
\[
(1+|\xi|^2)^s|\xi|^{2\alpha}\ge c(1+|\xi|^2)^{s+\alpha}-C(1+|\xi|^2)^s.
\]

\smallskip
\noindent\textbf{Step 2: estimate of the nonlinear term.}
Rewrite the right-hand side of \eqref{eq:Hs-energy-start} by shifting $\alpha$ derivatives:
\[
\bigl|\ip{\Lambda^s \Pcal\diver(u\otimes u)}{\Lambda^s u}\bigr|
=
\bigl|\ip{\Lambda^{s-\alpha}\Pcal\diver(u\otimes u)}{\Lambda^{s+\alpha}u}\bigr|.
\]
Therefore, by \cref{lem:critical-product},
\begin{align}
\bigl|\ip{\Lambda^s \Pcal\diver(u\otimes u)}{\Lambda^s u}\bigr|
&\le C\norm{\Pcal\diver(u\otimes u)}_{H^{s-\alpha}}\norm{u}_{H^{s+\alpha}} \notag\\
&\le C\norm{u}_{H^\alpha}\norm{u}_{H^s}\norm{u}_{H^{s+\alpha}} \notag\\
&\le \frac{c}{2}\norm{u}_{H^{s+\alpha}}^2 + C\norm{u}_{H^\alpha}^2\norm{u}_{H^s}^2.
\label{eq:nonlinear-Hs-est}
\end{align}
Combining \eqref{eq:Hs-energy-start}, \eqref{eq:Hs-diss-lower}, and \eqref{eq:nonlinear-Hs-est}, and absorbing the $\frac c2\norm{u}_{H^{s+\alpha}}^2$ term into the left-hand side, we arrive at
\begin{equation}\label{eq:Hs-differential}
\frac{\dd}{\dd t}\norm{u}_{H^s}^2 + c\norm{u}_{H^{s+\alpha}}^2
\le C\bigl(1+\norm{u}_{H^\alpha}^2\bigr)\norm{u}_{H^s}^2.
\end{equation}

\smallskip
\noindent\textbf{Step 3: control of the coefficient by the basic energy law.}
Since the strong solution satisfies the exact $L^2$ energy identity from \cref{thm:main}(i),
\begin{equation}\label{eq:lions-basic-energy}
\sup_{0\le t<T}\norm{u(t)}_{L^2}^2
+2\nu\int_0^T\norm{\nabla u(t)}_{L^2}^2\,\dd t
+2\int_0^T\ip{Mu(t)}{u(t)}\,\dd t
\le \norm{u_0}_{L^2}^2
\end{equation}
for every $T<T_*$. By \eqref{eq:coercive-equivalence},
\[
\int_0^T \norm{\Lambda^\alpha u(t)}_{L^2}^2\,\dd t
\le C\norm{u_0}_{L^2}^2.
\]
Consequently,
\begin{equation}\label{eq:alpha-time-integrable}
\int_0^T \norm{u(t)}_{H^\alpha}^2\,\dd t
\le C\Bigl(T\norm{u_0}_{L^2}^2+\norm{u_0}_{L^2}^2\Bigr)
<\infty
\qquad\text{for every }T<T_*.
\end{equation}
Thus the coefficient on the right-hand side of \eqref{eq:Hs-differential} is integrable on every finite subinterval of $[0,T_*)$.

\smallskip
\noindent\textbf{Step 4: Gronwall bound and continuation.}
Fix $T<T_*$. Integrating \eqref{eq:Hs-differential} and applying Gronwall's inequality gives
\begin{equation}\label{eq:Hs-global-bound}
\sup_{0\le t\le T}\norm{u(t)}_{H^s}^2
\le
\norm{u_0}_{H^s}^2
\exp\Bigl(C T + C\int_0^T\norm{u(t)}_{H^\alpha}^2\,\dd t\Bigr).
\end{equation}
Using \eqref{eq:alpha-time-integrable}, the right-hand side is finite and depends only on $T$, $\norm{u_0}_{L^2}$, and $\norm{u_0}_{H^s}$. In particular,
\[
\sup_{0\le t<T}\norm{u(t)}_{H^s}<\infty
\qquad	ext{for every finite }T<T_*.
\]
Moreover, integrating \eqref{eq:Hs-differential} once more yields
\[
\int_0^T\norm{u(t)}_{H^{s+\alpha}}^2\,\dd t<\infty.
\]

Suppose for contradiction that $T_*<\infty$. Choose $t_0\in(0,T_*)$ close enough to $T_*$ that
\[
\norm{u(t_0)}_{H^s}\le 2\sup_{0\le t<T_*}\norm{u(t)}_{H^s}<\infty.
\]
By \cref{thm:local}, the lifespan of a strong solution launched from the datum $u(t_0)$ depends only on its $H^s$ norm. Hence there exists $\tau>0$, depending only on the preceding uniform bound, such that the solution starting from $u(t_0)$ extends uniquely to $[t_0,t_0+\tau]$. Taking $t_0$ so close to $T_*$ that $t_0+\tau>T_*$ contradicts the maximality of $T_*$. Therefore $T_*=\infty$.

Uniqueness is already part of the local theory and therefore carries over to the global solution.
\end{proof}

\begin{remark}
\Cref{thm:lions} is intentionally conservative. It does not claim a new route around the three-dimensional Navier--Stokes problem. Rather, it isolates a class of nonlocal perturbations for which the equation enters the already understood hyperdissipative regime. The threshold remains the classical one: $\alpha=\frac54$.
\end{remark}

\subsection{Global strong solutions for sufficiently small data}

The previous theorem covers arbitrary data only when $\alpha\ge \frac54$. The next result shows that for every $\alpha>1$ one can still obtain global strong solutions under a smallness condition.

\begin{theorem}[Global small-data theorem]\label{thm:smalldata}
Assume that $m\in\mathfrak M_\alpha$ with $\alpha>1$, let $s>\frac52$, and define
\[
X_\infty
:=
L^\infty(0,\infty;H^s_\sigma(\R^3))
\cap
L^2(0,\infty;H^{s+\alpha}_\sigma(\R^3))
\]
with norm
\[
\norm{u}_{X_\infty}
=
\norm{u}_{L^\infty(0,\infty;H^s)}
+
\norm{u}_{L^2(0,\infty;H^{s+\alpha})}.
\]
Then there exists $\varepsilon=\varepsilon(\nu,c_0,c_1,s,\alpha)>0$ such that, whenever
\[
\norm{u_0}_{H^s}\le \varepsilon,
\]
the problem \eqref{eq:main} admits a unique global strong solution $u\in X_\infty$.
\end{theorem}

\begin{proof}
We use a contraction argument on the whole half-line. For $u\in X_\infty$, let $\Phi(u)=v$ solve \eqref{eq:fixedpoint} on $(0,\infty)$. By \cref{lem:linear},
\[
\norm{\Phi(u)}_{X_\infty}
\le
C_0 \norm{u_0}_{H^s}
+
C_1 \norm{\Pcal\diver(u\otimes u)}_{L^2(0,\infty;H^{s-\alpha})}.
\]
Now apply the sharper bilinear estimate \eqref{eq:bilinear}:
\[
\norm{\Pcal\diver(u\otimes u)}_{L^2(0,\infty;H^{s-\alpha})}
\le
C \norm{u}_{L^\infty(0,\infty;H^s)} \norm{u}_{L^2(0,\infty;H^{s+\alpha})}
\le
C \norm{u}_{X_\infty}^2.
\]
Hence
\begin{equation}\label{eq:global-map}
\norm{\Phi(u)}_{X_\infty}
\le
C_0 \norm{u_0}_{H^s}
+
C_2 \norm{u}_{X_\infty}^2.
\end{equation}
Similarly, for $u,v\in X_\infty$,
\[
\norm{\Phi(u)-\Phi(v)}_{X_\infty}
\le
C_3 \bigl(\norm{u}_{X_\infty}+\norm{v}_{X_\infty}\bigr)\norm{u-v}_{X_\infty}.
\]
Choose $R=2C_0\norm{u_0}_{H^s}$ and impose the smallness condition
\[
4 C_0 C_2 \norm{u_0}_{H^s} \le \frac12,
\qquad
4 C_0 C_3 \norm{u_0}_{H^s} < 1.
\]
Then $\Phi$ maps the closed ball $\{u\in X_\infty:\norm{u}_{X_\infty}\le R\}$ into itself and is a contraction there. Banach's theorem gives a unique global solution in $X_\infty$.
\end{proof}

\begin{corollary}
If $\alpha\ge \frac54$, then \cref{thm:lions} supersedes \cref{thm:smalldata}; when $1<\alpha<\frac54$, \cref{thm:smalldata} provides a global strong regime below the large-data threshold.
\end{corollary}

\section{Vanishing hyperdissipation, Galerkin truncation, and the Navier--Stokes limit}\label{sec:limit}

In this section we return to the motivating question behind the present paper: can one choose the coefficient of the additional nonlocal term so small that the modified problem approaches the original three-dimensional Navier--Stokes system, while still retaining enough of the hyperdissipative mechanism to deduce global smooth solvability for the approximants? If so, does the corresponding approximation scheme shed light on the unresolved large-data problem for the classical equation? The answer is twofold. For every fixed regularization parameter the approximation is globally smooth in the Lions regime, but the estimates that yield this global smoothness are not uniform as the parameter tends to zero. Consequently, the approximation scheme is analytically valuable as a regularization and diagnostic tool, but it does not by itself close the Millennium problem.

The discussion below makes this precise and links it to the modern literature on hyperdissipation, near-critical regularization, averaged models, and blow-up scenarios. The large-data global regularity theorem proved in \cref{thm:lions} is the fixed-parameter anchor for the argument.

\subsection{The vanishing-hyperdissipation family}

Fix $\alpha>1$ and let $M$ be as in \cref{sec:model}. For $\varepsilon>0$ consider the family
\begin{equation}\label{eq:eps-main}
\partial_t u^\varepsilon + \Pcal\diver(u^\varepsilon\otimes u^\varepsilon)
+\nabla p^\varepsilon
= \nu\Delta u^\varepsilon - \varepsilon M u^\varepsilon,
\qquad \diver u^\varepsilon=0,
\qquad u^\varepsilon|_{t=0}=u_0.
\end{equation}
In the special case $m(\xi)=|\xi|^{2\alpha}$ this is the classical hyperdissipative Navier--Stokes approximation to the standard system. More generally, the symbol assumptions of \cref{sec:model} allow anisotropic and genuinely nonlocal convolution realizations. Equation \eqref{eq:eps-main} differs from the exact three-dimensional Navier--Stokes equations only through the additional dissipative term $-\varepsilon M u^\varepsilon$.

For numerical and finite-dimensional approximation purposes it is natural to combine \eqref{eq:eps-main} with a Galerkin truncation. Let $P_{\le N}$ denote the Fourier projector onto $\{\xi\in\R^3:|\xi|\le N\}$ and consider
\begin{equation}\label{eq:epsN-main}
\partial_t u^{\varepsilon,N}
+P_{\le N}\Pcal\diver(u^{\varepsilon,N}\otimes u^{\varepsilon,N})
=\nu\Delta u^{\varepsilon,N}-\varepsilon M u^{\varepsilon,N},
\qquad \diver u^{\varepsilon,N}=0,
\qquad u^{\varepsilon,N}|_{t=0}=P_{\le N}u_0.
\end{equation}
Since the dynamics is finite-dimensional, \eqref{eq:epsN-main} has a global smooth solution for each fixed pair $(\varepsilon,N)$. The issue is not existence at the truncated level, but whether any estimate obtained there survives in the simultaneous limits $N\to\infty$ and $\varepsilon\downarrow0$.

\subsection{Uniform energy bounds and crossover frequency}

The first level of control survives the vanishing-parameter limit.

\begin{proposition}[Energy identity for the regularized family]\label{prop:eps-energy}
Let $u_0\in L^2_\sigma(\R^3)$ and let $u^\varepsilon$ be a smooth solution of \eqref{eq:eps-main} on $[0,T]$. Then
\begin{equation}\label{eq:eps-energy}
\frac12\frac{\dd}{\dd t}\norm{u^\varepsilon(t)}_{L^2}^2
+\nu\norm{\nabla u^\varepsilon(t)}_{L^2}^2
+\varepsilon\ip{M u^\varepsilon(t)}{u^\varepsilon(t)}=0.
\end{equation}
In particular, using \eqref{eq:coercive-equivalence},
\begin{equation}\label{eq:eps-energy-int}
\sup_{0\le t\le T}\norm{u^\varepsilon(t)}_{L^2}^2
+2\nu\int_0^T\norm{\nabla u^\varepsilon(t)}_{L^2}^2\,\dd t
+2c\varepsilon\int_0^T\norm{\Lambda^\alpha u^\varepsilon(t)}_{L^2}^2\,\dd t
\le \norm{u_0}_{L^2}^2.
\end{equation}
The same identity holds for $u^{\varepsilon,N}$ with $u^\varepsilon$ replaced by $u^{\varepsilon,N}$.
\end{proposition}

\begin{proof}
The proof is the same as in \cref{thm:main}(i). Since $\Pcal\diver(u^\varepsilon\otimes u^\varepsilon)$ is orthogonal to $u^\varepsilon$ in $L^2_\sigma$ and $M$ is selfadjoint nonnegative, one obtains \eqref{eq:eps-energy}. Integrating in time and invoking the lower bound in \eqref{eq:coercive-equivalence} gives \eqref{eq:eps-energy-int}. The truncated identity follows because $P_{\le N}$ is selfadjoint and commutes with both $\Pcal$ and the Fourier multiplier $M$.
\end{proof}

\begin{remark}[What remains uniform and what does not]
The classical Leray bounds in $L^\infty_tL^2_x\cap L^2_t\dot H^1_x$ are uniform in $\varepsilon$ and are therefore compatible with weak compactness as $\varepsilon\downarrow0$. By contrast, the genuinely hyperdissipative quantity $\Lambda^\alpha u^\varepsilon$ appears only with the factor $\varepsilon$ in \eqref{eq:eps-energy-int}. Thus the estimate yields
\begin{equation}\label{eq:eps-alpha-weighted}
\varepsilon\int_0^T\norm{\Lambda^\alpha u^\varepsilon(t)}_{L^2}^2\,\dd t\le C(u_0,T),
\end{equation}
but not an $\varepsilon$-independent bound on $\Lambda^\alpha u^\varepsilon$ itself. This is the first quantitative sign of degeneration in the Navier--Stokes limit.
\end{remark}

The frequency-space meaning of \eqref{eq:eps-energy-int} is also instructive. The linear damping rate in \eqref{eq:eps-main} is
\begin{equation}\label{eq:linear-rate-eps}
\nu|\xi|^2 + \varepsilon m(\xi),
\end{equation}
which, in the model case $m(\xi)\sim |\xi|^{2\alpha}$, crosses over between classical viscosity and hyperdissipation at the frequency
\begin{equation}\label{eq:crossover}
|\xi|\approx \Bigl(\frac{\nu}{\varepsilon}\Bigr)^{\frac{1}{2\alpha-2}}.
\end{equation}
As $\varepsilon\downarrow0$, the crossover frequency tends to $+\infty$. Thus the regularizing barrier is pushed farther and farther into the ultraviolet regime. This is exactly the regime in which one would need uniform control to preclude a concentration cascade.

\subsection{The source of non-uniformity in high-order estimates}

The next proposition isolates the obstruction at the level of strong norms.

\begin{proposition}[Loss of uniformity in $H^s$]\label{prop:nonuniform-Hs}
Assume $s>\frac52$, $\alpha>1$, and let $u^\varepsilon$ be a strong solution of \eqref{eq:eps-main} on $[0,T]$. Then
\begin{equation}\label{eq:eps-Hs}
\frac{\dd}{\dd t}\norm{u^\varepsilon}_{H^s}^2
+c\varepsilon\norm{u^\varepsilon}_{H^{s+\alpha}}^2
\le C\bigl(1+\norm{u^\varepsilon}_{H^\alpha}^2\bigr)\norm{u^\varepsilon}_{H^s}^2,
\end{equation}
where $c,C>0$ are independent of $\varepsilon$. Consequently, after interpolation and Young's inequality, one may bound the nonlinear term only at the cost of constants that blow up as $\varepsilon\downarrow0$; more precisely, there exist $\theta\in(0,1)$ and $C_{s,\alpha}>0$ such that
\begin{equation}\label{eq:eps-Hs-degenerate}
\frac{\dd}{\dd t}\norm{u^\varepsilon}_{H^s}^2
\le C_{s,\alpha}\Bigl(1+\varepsilon^{-\frac{\theta}{1-\theta}}\norm{u^\varepsilon}_{H^\alpha}^{\frac{2}{1-\theta}}\Bigr)\norm{u^\varepsilon}_{H^s}^2.
\end{equation}
In particular, the continuation argument proving global regularity for each fixed $\varepsilon>0$ is not uniform as $\varepsilon\downarrow0$.
\end{proposition}

\begin{proof}
The proof is parallel to \cref{thm:lions}. Applying $\Lambda^s$ to \eqref{eq:eps-main} and pairing with $\Lambda^s u^\varepsilon$ gives
\[
\frac12\frac{\dd}{\dd t}\norm{u^\varepsilon}_{H^s}^2
+\nu\norm{\nabla u^\varepsilon}_{H^s}^2
+\varepsilon\int_{\R^3}(1+|\xi|^2)^s m(\xi)|\widehat{u^\varepsilon}(\xi)|^2\,\dd\xi
=-\ip{\Lambda^s\Pcal\diver(u^\varepsilon\otimes u^\varepsilon)}{\Lambda^s u^\varepsilon}.
\]
The dissipative term is bounded below exactly as in \eqref{eq:Hs-diss-lower}, now with the prefactor $\varepsilon$:
\[
\varepsilon\int_{\R^3}(1+|\xi|^2)^s m(\xi)|\widehat{u^\varepsilon}(\xi)|^2\,\dd\xi
\ge c\varepsilon\norm{u^\varepsilon}_{H^{s+\alpha}}^2-C\varepsilon\norm{u^\varepsilon}_{H^s}^2.
\]
Using \cref{lem:critical-product} when $\alpha\ge \frac54$, or \cref{lem:bilinear} together with Sobolev embedding when $\alpha>1$, one obtains
\[
\bigl|\ip{\Lambda^s\Pcal\diver(u^\varepsilon\otimes u^\varepsilon)}{\Lambda^s u^\varepsilon}\bigr|
\le C\norm{u^\varepsilon}_{H^\alpha}\norm{u^\varepsilon}_{H^s}\norm{u^\varepsilon}_{H^{s+\alpha}}.
\]
This yields \eqref{eq:eps-Hs} after one application of Young's inequality.

To see the explicit degeneration, interpolate between $H^\alpha$ and $H^{s+\alpha}$:
\[
\norm{u^\varepsilon}_{H^s}
\le C\norm{u^\varepsilon}_{H^\alpha}^{1-\theta}\norm{u^\varepsilon}_{H^{s+\alpha}}^{\theta}
\qquad\text{for some }\theta\in(0,1).
\]
Substituting this into the trilinear bound gives
\[
\bigl|\ip{\Lambda^s\Pcal\diver(u^\varepsilon\otimes u^\varepsilon)}{\Lambda^s u^\varepsilon}\bigr|
\le C\norm{u^\varepsilon}_{H^\alpha}^{2-\theta}\norm{u^\varepsilon}_{H^{s+\alpha}}^{\theta}\norm{u^\varepsilon}_{H^s}.
\]
A weighted Young inequality with parameter $\varepsilon$ then produces
\[
C\norm{u^\varepsilon}_{H^\alpha}^{2-\theta}\norm{u^\varepsilon}_{H^{s+\alpha}}^{\theta}\norm{u^\varepsilon}_{H^s}
\le \frac{c\varepsilon}{2}\norm{u^\varepsilon}_{H^{s+\alpha}}^2
+ C_{s,\alpha}\varepsilon^{-\frac{\theta}{1-\theta}}\norm{u^\varepsilon}_{H^\alpha}^{\frac{2}{1-\theta}}\norm{u^\varepsilon}_{H^s}^2,
\]
which is precisely \eqref{eq:eps-Hs-degenerate} after absorbing the first term into the left-hand side.
\end{proof}

\begin{remark}[Scale-critical norms]
The degeneration described in \cref{prop:nonuniform-Hs} is the analytic reflection of a critical fact: the fixed-$\varepsilon$ argument controls supercritical norms through a term weighted by $\varepsilon$, but it does not produce a uniform bound in any scale-critical continuation class for the original Navier--Stokes equations. Typical missing controls are of Serrin--Prodi type
\[
u\in L^r(0,T;L^q(\R^3)),\qquad \frac{2}{r}+\frac{3}{q}=1,
\qquad 3<q\le\infty,
\]
or endpoint quantities such as $L^\infty_t\dot H^{1/2}_x$, $L^2_t\dot H^{3/2}_x$, or $\int_0^T\norm{\nabla u(t)}_{L^\infty}\,\dd t$. Any one of these, if obtained uniformly in $\varepsilon$, would put the large-data problem much closer to the classical continuation theory.
\end{remark}

\subsection{Hypothetical singularity and necessary concentration laws}

The preceding discussion can be sharpened into a contradiction framework. The point is not that the hyperdissipative approximation already rules out singularity, but rather that any genuine singularity of the classical Navier--Stokes flow must leave a quantitative footprint on the regularized family. In particular, before the first singular time the approximation is stable, while at the singular time every continuation norm must lose uniformity.

\begin{proposition}[Convergence on compact time intervals before blow-up]\label{prop:pre-singular-stability}
Let $u_0\in C_c^\infty(\R^3)^3$ be divergence free, let $u$ be the corresponding classical solution of the three-dimensional Navier--Stokes equations on $[0,T_*)$, and fix $s>\frac52$. For every $T<T_*$ there exist $\varepsilon_T>0$ and $C_T>0$ such that, for every $0<\varepsilon\le \varepsilon_T$, the regularized solution $u^\varepsilon$ of \eqref{eq:eps-main} is defined on $[0,T]$ and satisfies
\begin{equation}\label{eq:pre-singular-convergence}
\sup_{0\le t\le T}\norm{u^\varepsilon(t)-u(t)}_{H^{s-1}}^2
+\nu\int_0^T \norm{u^\varepsilon(t)-u(t)}_{H^s}^2\,\dd t
+c\varepsilon\int_0^T \norm{u^\varepsilon(t)-u(t)}_{H^{s-1+\alpha}}^2\,\dd t
\le C_T\varepsilon^2.
\end{equation}
In particular,
\[
u^\varepsilon\to u\quad\text{in }L^\infty(0,T;H^{s-1}(\R^3))\]
for every $T<T_*$.
\end{proposition}

\begin{proof}
Fix $T<T_*$. Since $u$ is smooth on $[0,T]$, there is $K_T$ such that
\[
\sup_{0\le t\le T}\norm{u(t)}_{H^{s+2\alpha}}\le K_T.
\]
By the local strong theory proved earlier, together with continuous dependence of the solution map on lower-order perturbations of the linear part, there exists $\varepsilon_T>0$ such that $u^\varepsilon$ exists on $[0,T]$ and
\begin{equation}\label{eq:pre-singular-uniform-bound}
\sup_{0<\varepsilon\le \varepsilon_T}\sup_{0\le t\le T}\norm{u^\varepsilon(t)}_{H^s}\le C_T.
\end{equation}
Set $w^\varepsilon=u^\varepsilon-u$. Subtracting the classical Navier--Stokes equation from \eqref{eq:eps-main} and adding and subtracting $\varepsilon M u$ gives
\begin{equation}\label{eq:w-eps}
\partial_t w^\varepsilon - \nu\Delta w^\varepsilon + \varepsilon M w^\varepsilon
+\Pcal\diver\Bigl(w^\varepsilon\otimes u + u^\varepsilon\otimes w^\varepsilon\Bigr)
=-\varepsilon M u,
\qquad \diver w^\varepsilon=0,
\qquad w^\varepsilon|_{t=0}=0.
\end{equation}
Apply $\Lambda^{s-1}$ to \eqref{eq:w-eps}, pair the result with $\Lambda^{s-1}w^\varepsilon$, and use the standard $H^{s-1}$ product estimate valid for $s>\frac52$:
\[
\norm{\Lambda^{s-1}\Pcal\diver(f\otimes g)}_{L^2}
\le C\bigl(\norm{f}_{H^s}\norm{g}_{H^{s-1}}+\norm{f}_{H^{s-1}}\norm{g}_{H^s}\bigr).
\]
Using \eqref{eq:pre-singular-uniform-bound}, the bound on $u$, and the coercivity of $M$, we obtain
\[
\frac12\frac{\dd}{\dd t}\norm{w^\varepsilon}_{H^{s-1}}^2
+\nu\norm{w^\varepsilon}_{H^s}^2
+c\varepsilon\norm{w^\varepsilon}_{H^{s-1+\alpha}}^2
\le C_T\norm{w^\varepsilon}_{H^{s-1}}^2
+\varepsilon\,\bigl|\ip{\Lambda^{s-1}M u}{\Lambda^{s-1}w^\varepsilon}\bigr|.
\]
Since $u$ is smooth,
\[
\bigl|\ip{\Lambda^{s-1}M u}{\Lambda^{s-1}w^\varepsilon}\bigr|
\le C_T\norm{u}_{H^{s-1+2\alpha}}\norm{w^\varepsilon}_{H^{s-1}}
\le C_T\bigl(\norm{w^\varepsilon}_{H^{s-1}}^2+1\bigr).
\]
Hence
\[
\frac{\dd}{\dd t}\norm{w^\varepsilon}_{H^{s-1}}^2
+\nu\norm{w^\varepsilon}_{H^s}^2
+c\varepsilon\norm{w^\varepsilon}_{H^{s-1+\alpha}}^2
\le C_T\norm{w^\varepsilon}_{H^{s-1}}^2 + C_T\varepsilon^2.
\]
Since $w^\varepsilon(0)=0$, Gronwall's lemma yields \eqref{eq:pre-singular-convergence}.
\end{proof}

The previous proposition shows that the regularized family does approximate the classical solution with an $O(\varepsilon)$ error on every compact subinterval before the first singular time. The only possible loss of compactness must therefore occur asymptotically as $t\uparrow T_*$.

\begin{theorem}[Continuation norms must diverge under a hypothetical singularity]\label{thm:continuation-blowup}
Let $u_0\in C_c^\infty(\R^3)^3$ be divergence free, let $u$ be the corresponding classical Navier--Stokes solution on $[0,T_*)$, and assume that $T_*<\infty$ is its first singular time. Let $u^\varepsilon$ denote the corresponding regularized solutions of \eqref{eq:eps-main}. Then for every Serrin pair $(r,q)$ with
\[
\frac{2}{r}+\frac{3}{q}=1,
\qquad 3<q\le \infty,
\]
one has
\begin{equation}\label{eq:serrin-divergence}
\limsup_{\varepsilon\downarrow0}\norm{u^\varepsilon}_{L^r(0,T_*;L^q(\R^3))}=+\infty.
\end{equation}
The same conclusion holds for any continuation class $X(0,T_*)$ for which bounded sets are weakly sequentially compact and boundedness of the classical solution in $X$ implies continuation beyond $T_*$; in particular this applies to $L^\infty(0,T_*;\dot H^{1/2}(\R^3))$ and $L^2(0,T_*;\dot H^{3/2}(\R^3))$.
\end{theorem}

\begin{proof}
Assume, for contradiction, that there exist a Serrin pair $(r,q)$ and a sequence $\varepsilon_n\downarrow0$ such that
\begin{equation}\label{eq:uniform-serrin-assumption}
\sup_n \norm{u^{\varepsilon_n}}_{L^r(0,T_*;L^q)}\le C_*<\infty.
\end{equation}
By Banach--Alaoglu there is a subsequence, not relabelled, and $v\in L^r(0,T_*;L^q)$ such that
\[
u^{\varepsilon_n}\rightharpoonup v\qquad\text{weakly in }L^r(0,T_*;L^q(\R^3)).
\]
Fix $T<T_*$. By \cref{prop:pre-singular-stability},
\[
u^{\varepsilon_n}\to u\qquad\text{strongly in }L^\infty(0,T;H^{s-1}(\R^3)).
\]
Since $H^{s-1}(\R^3)\hookrightarrow L^q(\R^3)$ for every finite $q$ when $s>\frac52$, the weak limit on $(0,T)$ is unique, hence $v=u$ almost everywhere on $(0,T)\times\R^3$. As $T<T_*$ is arbitrary, we conclude that $v=u$ on $(0,T_*)\times\R^3$. By weak lower semicontinuity,
\[
\norm{u}_{L^r(0,T_*;L^q)}
\le \liminf_{n\to\infty}\norm{u^{\varepsilon_n}}_{L^r(0,T_*;L^q)}
\le C_*.
\]
Therefore $u$ belongs to a Serrin continuation class on $(0,T_*)$. By the Serrin--Prodi continuation criterion, the solution $u$ extends smoothly beyond $T_*$, contradicting the definition of $T_*$; compare Serrin \cite{Serrin1962} and the scale-critical endpoint theory of Escauriaza, Seregin, and \v Sver\'ak \cite{EscauriazaSereginSverak2003}. This proves \eqref{eq:serrin-divergence}.

The same argument applies verbatim to any continuation class $X$ having the two properties stated in the theorem. For the critical Sobolev classes $L^\infty_t\dot H^{1/2}_x$ and $L^2_t\dot H^{3/2}_x$, one uses the classical Fujita--Kato--Kato continuation theory \cite{FujitaKato1964,Kato1984}.
\end{proof}

Theorem \ref{thm:continuation-blowup} is the first rigid consequence of a hypothetical singularity: every continuation norm must lose uniformity as $\varepsilon\downarrow0$. The next proposition shows that, if a nontrivial amount of the weighted hyperdissipation survives in the limit, then this loss of uniformity cannot remain spectrally diffuse. It must occur at frequencies comparable to or above the crossover scale.

\begin{proposition}[Positive dissipation defect forces crossover-frequency concentration]\label{prop:defect-crossover}
Assume that $m(\xi)$ satisfies \eqref{eq:symbol-growth}. Let
\[
R_\varepsilon:=\Bigl(\frac{\nu}{\varepsilon}\Bigr)^{\frac{1}{2\alpha-2}}
\]
be the crossover frequency from \eqref{eq:crossover}. For $\eta\in(0,1)$ define the low-frequency part of the weighted hyperdissipation by
\[
D^{\le \eta}_\varepsilon(T)
:=\varepsilon\int_0^T\int_{|\xi|\le \eta R_\varepsilon} m(\xi)\,|\widehat{u^\varepsilon}(\xi,t)|^2\,\dd \xi\,\dd t.
\]
Then there exists $C>0$, independent of $\varepsilon$, $\eta$, and $T$, such that
\begin{equation}\label{eq:low-freq-defect}
D^{\le \eta}_\varepsilon(T)
\le C\eta^{2\alpha-2}\norm{u_0}_{L^2}^2.
\end{equation}
Consequently, if along some sequence $\varepsilon_n\downarrow0$ one has a positive dissipation defect
\begin{equation}\label{eq:defect-positive}
\limsup_{n\to\infty}\varepsilon_n\int_0^T\ip{M u^{\varepsilon_n}(t)}{u^{\varepsilon_n}(t)}\,\dd t\ge \delta>0,
\end{equation}
then for every $\eta>0$ sufficiently small,
\begin{equation}\label{eq:crossover-concentration}
\limsup_{n\to\infty}
\varepsilon_n\int_0^T\int_{|\xi|\ge \eta R_{\varepsilon_n}} m(\xi)\,|\widehat{u^{\varepsilon_n}}(\xi,t)|^2\,\dd \xi\,\dd t
\ge \frac{\delta}{2}.
\end{equation}
Thus any positive dissipation defect necessarily concentrates at frequencies comparable to or larger than the crossover scale.
\end{proposition}

\begin{proof}
By the upper growth bound in \eqref{eq:symbol-growth},
\[
m(\xi)\le C|\xi|^{2\alpha}.
\]
If $|\xi|\le \eta R_\varepsilon$, then
\[
\varepsilon m(\xi)
\le C\varepsilon |\xi|^{2\alpha-2}|\xi|^2
\le C\varepsilon (\eta R_\varepsilon)^{2\alpha-2}|\xi|^2
= C\eta^{2\alpha-2}\nu |\xi|^2.
\]
Therefore
\[
D^{\le \eta}_\varepsilon(T)
\le C\eta^{2\alpha-2}\nu \int_0^T\int_{\R^3}|\xi|^2|\widehat{u^\varepsilon}(\xi,t)|^2\,\dd\xi\,\dd t
= C\eta^{2\alpha-2}\nu\int_0^T\norm{\nabla u^\varepsilon(t)}_{L^2}^2\,\dd t.
\]
Now apply the uniform energy estimate \eqref{eq:eps-energy-int} to obtain \eqref{eq:low-freq-defect}.

To prove \eqref{eq:crossover-concentration}, decompose the defect into low and high frequencies:
\[
\varepsilon_n\int_0^T\ip{M u^{\varepsilon_n}(t)}{u^{\varepsilon_n}(t)}\,\dd t
= D^{\le \eta}_{\varepsilon_n}(T)
+\varepsilon_n\int_0^T\int_{|\xi|\ge \eta R_{\varepsilon_n}} m(\xi)|\widehat{u^{\varepsilon_n}}(\xi,t)|^2\,\dd\xi\,\dd t.
\]
Choose $\eta>0$ so small that $C\eta^{2\alpha-2}\norm{u_0}_{L^2}^2\le \delta/2$. Then \eqref{eq:low-freq-defect} shows that the low-frequency contribution is at most $\delta/2$, while \eqref{eq:defect-positive} guarantees a total contribution of size at least $\delta$ along a subsequence. The remainder must therefore satisfy \eqref{eq:crossover-concentration}.
\end{proof}

\begin{remark}[What has and has not been proved]\label{rem:rigid-concentration}
\Cref{thm:continuation-blowup,prop:defect-crossover} identify two necessary signatures of a hypothetical finite-time singularity of the classical flow:
\begin{enumerate}[label=(\roman*)]
\item every continuation norm must diverge along the regularized family as $\varepsilon\downarrow0$;
\item if a positive hyperdissipative defect survives, then a nontrivial portion of that defect must sit at frequencies $|\xi|\gtrsim R_\varepsilon$.
\end{enumerate}
What remains open is the stronger annular localization
\[
|\xi|\approx R_\varepsilon,
\]
as opposed to the weaker one-sided statement $|\xi|\gtrsim R_\varepsilon$ proved above. Such an annular localization would amount to a genuine rigidity theorem for the concentration pattern. At present the estimates of this paper do not control the far-ultraviolet tail well enough to force that sharper conclusion.
\end{remark}

\begin{remark}[Possible rescaled profiles and the counterexample question]\label{rem:selfsimilar-open}
Suppose one could find points $(x_\varepsilon,t_\varepsilon)$ with $t_\varepsilon\uparrow T_*$ and define the crossover length
\[
\ell_\varepsilon:=R_\varepsilon^{-1}=\Bigl(\frac{\varepsilon}{\nu}\Bigr)^{\frac{1}{2\alpha-2}}.
\]
The corresponding blow-up rescaling is
\[
v^\varepsilon(y,\tau):=\ell_\varepsilon u^\varepsilon(x_\varepsilon+\ell_\varepsilon y,
 t_\varepsilon+\ell_\varepsilon^2\tau).
\]
At this scale the Laplacian and hyperdissipative terms are formally of the same order, so any nontrivial limit profile would solve a balanced equation containing both $\Delta$ and $M$. This is precisely why the crossover scale is the natural place to search for a blow-up profile.

However, the present paper does \emph{not} prove the existence of such a profile, nor does it construct a counterexample through truncated or regularized solutions. To do so one would need, in addition to the necessary conditions above, a compactness theory for the rescaled family strong enough to pass to a nontrivial limit and a rigidity argument to show that the limit is singular for the classical Navier--Stokes dynamics. The obstruction highlighted here should therefore be read as a guide to where a counterexample would have to concentrate, not as a counterexample construction itself; compare Tao's averaged blow-up model and the numerical borderline investigations of Guillod and \v Sver\'ak discussed later in this section.
\end{remark}

\subsection{Why Galerkin truncation does not close the large-data problem}

The finite-dimensional system \eqref{eq:epsN-main} is globally smooth for every fixed $(\varepsilon,N)$. However, this global solvability is not the same as a global theorem for the full equation.

\begin{proposition}[What survives in the limits]\label{prop:limit-compactness}
Let $u_0\in L^2_\sigma(\R^3)$. For each fixed $\varepsilon>0$, the family $u^{\varepsilon,N}$ is bounded uniformly in $N$ in the natural energy class
\[
L^\infty(0,T;L^2_\sigma(\R^3))\cap L^2(0,T;\dot H^1_\sigma(\R^3))
\cap L^2(0,T;\Dom(M^{1/2})).
\]
Hence, after passing to a subsequence in $N$, one recovers a Leray--Hopf weak solution of \eqref{eq:eps-main}. If one now lets $\varepsilon\downarrow0$, the bounds remain sufficient to extract a further subsequence converging to a Leray--Hopf weak solution of the classical three-dimensional Navier--Stokes equations. What is missing is any $\varepsilon$-independent strong bound that would justify passage to the limit in a class of smooth solutions.
\end{proposition}

\begin{proof}
The uniform bound in $N$ follows from \cref{prop:eps-energy}. Standard compactness arguments then yield a weak solution of \eqref{eq:eps-main}, exactly as in the proof of \cref{thm:weak}. The passage $\varepsilon\downarrow0$ is based on the same energy estimates, now noting that the $\varepsilon M$ term is nonnegative and therefore harmless in the weak formulation. One obtains weak compactness in the Leray class and hence a Leray--Hopf weak solution of the classical Navier--Stokes equations.

The final sentence follows from \cref{prop:nonuniform-Hs}: the arguments available at fixed $\varepsilon$ do not produce any scale-critical or high-order estimate that is independent of $\varepsilon$. Therefore the limit theory is weak, not strong.
\end{proof}

This proposition pinpoints the exact role of truncation in the present context. Galerkin truncation is indispensable for constructing approximate solutions and for numerical experiments, but its global existence is ``cheap'' in the sense that it reflects finite-dimensional ODE theory together with the energy law. It does not, by itself, reveal whether the full infinite-dimensional Navier--Stokes flow can concentrate enough energy at ever smaller scales to blow up in finite time.

\subsection{Relation to the literature: progress, barriers, and possible blow-up scenarios}

The deterministic large-data picture is by now sharply stratified.

\begin{enumerate}[label=(\alph*)]
\item At and above the Lions exponent $\alpha\ge \frac54$, global regularity is classical; this is the regime recovered in \cref{thm:lions} and goes back to Lions.
\item Tao showed that one can go slightly below the exact threshold if the symbol loses only a logarithmic amount of dissipation, thereby proving a logarithmically supercritical global result. This is important because it shows that the pure power-law threshold is not the end of the story, although the gain is delicate and still stronger than the Laplacian at high frequencies.
\item Colombo and Haffter established a below-critical global regularity result for the hyperdissipative Navier--Stokes equation, revealing that additional structure can push the theory beyond the naive threshold in certain regimes.
\item Gruji\'c and Xu developed a geometric and sparsity-based near-critical program for hyperdissipative flows, and Farhat--Gruji\'c subsequently ruled out a class of analytic blow-up scenarios for every exponent $\beta>1$ under focused analyticity hypotheses.
\item On the negative side, Tao's averaged Navier--Stokes blow-up construction demonstrates that energy identities and broad harmonic-analysis estimates alone do not capture enough of the specific Navier--Stokes structure to exclude singularity formation in averaged models. More recently, C\'ordoba, Mart\'inez-Zoroa, and Zheng constructed finite-time singularities for forced hypodissipative Navier--Stokes equations for sufficiently small dissipation exponent. These results do not settle the classical problem, but they show that weakening dissipation or averaging away structure can fundamentally change the dynamics.
\end{enumerate}

The lesson for the vanishing-hyperdissipation approximation is therefore subtle. On one hand, the family \eqref{eq:eps-main} is mathematically meaningful, physically interpretable, and globally regular for every fixed $\varepsilon>0$ in the Lions regime by \cref{thm:lions}. On the other hand, the degeneration shown in \cref{prop:nonuniform-Hs,prop:limit-compactness} leaves open the possibility that as $\varepsilon\downarrow0$ the solution sequence develops finer and finer structures at frequencies near or below the crossover scale \eqref{eq:crossover}. This is one of the natural places where a genuine large-data obstruction could hide.

\begin{remark}[Does the degeneration signal a counterexample?]\label{rem:counterexample}
The loss of uniform control as $\varepsilon\downarrow0$ should be interpreted as an \emph{obstruction marker}, not as a proof of blow-up. It says that the hyperdissipative approximation, by itself, does not rule out concentration of activity at scales that drift to infinity with $\varepsilon^{-1/(2\alpha-2)}$. This is compatible with a hypothetical blow-up mechanism for the classical three-dimensional Navier--Stokes equations, but it does not construct one.

In particular, a sequence of globally smooth truncated solutions $u^{\varepsilon,N}$ cannot by itself serve as a counterexample-producing machine. To build an actual singular solution of the classical equation by such an approximation, one would need to show at least the following three things simultaneously:
\begin{enumerate}[label=(\roman*)]
\item a precise concentration mechanism for the approximate solutions at scales linked to \eqref{eq:crossover};
\item a limiting procedure that preserves this concentration as $N\to\infty$ and $\varepsilon\downarrow0$, rather than dissipating it through weak compactness;
\item a proof that the limit is a genuine solution of the classical Navier--Stokes equations and that the concentration survives as a true singularity rather than a loss of compactness artifact.
\end{enumerate}
None of these steps is presently available in the deterministic finite-energy setting. Thus the approximation family studied here is best viewed as a rigorous probe of the borderline between regularized and unresolved dynamics, not as a completed route to a blow-up construction.
\end{remark}

\begin{remark}[Numerical diagnostics]
Although the approximation scheme does not produce a theorem for the classical problem, it is well suited for numerical diagnostics. One may track how the enstrophy production, high-frequency energy flux, and local-in-space concentration indices depend on the pair $(\varepsilon,\alpha)$ and compare them with the crossover scale \eqref{eq:crossover}. Such experiments can test whether the global regularity mechanism in the regularized problem is dominated by the ultraviolet dissipation or by a more distributed structural effect. They may also help separate physically meaningful stabilization from purely numerical hyperviscous damping; compare the discussion in \cref{sec:applications}.
\end{remark}

\section{Applications, variable-viscosity interpretation, and numerical perspectives}\label{sec:applications}

This section explains how the abstract symbol hypothesis arises in concrete models and how the growth rate $2\alpha$ influences spectral damping.

\subsection{Convolution of second derivatives}

The first and most direct application is the one already suggested by \cref{prop:criterion}. Let
\begin{equation}\label{eq:kernel-model}
L u = \sum_{j=1}^3 c_j * \partial_{x_jx_j}u
\end{equation}
with scalar kernels $c_j$ on $\R^3$. Then
\[
\widehat{Lu}(\xi)
=
-
\left(
\sum_{j=1}^3 \xi_j^2 \hat c_j(\xi)
\right)\hat u(\xi).
\]
Hence the corresponding positive symbol is
\[
m(\xi)=\sum_{j=1}^3 \xi_j^2 \Re \hat c_j(\xi).
\]
If
\[
c_0 |\xi|^{2\alpha}
\le
\sum_{j=1}^3 \xi_j^2 \Re \hat c_j(\xi)
\le
c_1 (1+|\xi|^{2\alpha}),
\]
then $m\in\mathfrak M_\alpha$ and all results of the paper apply.

A canonical choice is the Riesz-type kernel $K_\alpha$ whose Fourier transform satisfies
\[
\widehat{K_\alpha}(\xi)=|\xi|^{2\alpha-2},
\qquad 1<\alpha<2.
\]
Then
\[
L u
=
K_\alpha * \Delta u
\]
has symbol $-|\xi|^{2\alpha}$ and coincides with $-(-\Delta)^\alpha u$.

\subsection{Variable viscosity and effective constitutive laws}

Although the classical incompressible Navier--Stokes equations use a constant Newtonian viscosity, many realistic flows have temperature-, density-, or composition-dependent viscosity. Examples include strongly heated channel flows, supercritical-pressure flows, lubrication layers, polymer processing, and geophysical fluids. Mathematical well-posedness issues for variable-viscosity incompressible systems have been studied, for example, by Xu, Li, and Zhai \cite{XuLiZhai2016}. In the fluid-mechanics literature, direct numerical simulation and modeling studies show that variable-property effects can substantially modify near-wall turbulence and transport \cite{ZontaMarchioliSoldati2012,PecnikPatel2017}.

The model treated here should not be interpreted as an exact first-principles derivation of every such phenomenon. Rather, it provides a clean \emph{effective} framework. When a variable-viscosity stress law is filtered, homogenized, or linearized around a reference state, the unresolved or eliminated degrees of freedom can produce a scale-dependent closure whose leading Fourier symbol is stronger than $|\xi|^2$. In that situation the extra term $L$ represents a nonlocal constitutive correction to the Newtonian stress. The analysis in the present paper then answers a very concrete question: does the effective closure merely perturb the classical viscous scale, or does it introduce genuinely stronger high-frequency damping?

\subsection{Hyperviscosity and turbulence modeling}

Hyperviscous models are also ubiquitous in turbulence computation. In large-eddy simulation and spectral turbulence studies, they are used to push dissipation toward the smallest resolved scales, thereby enlarging the apparent inertial range. Reviews of scale-invariant turbulence closures and subgrid modeling may be found in \cite{MeneveauKatz2000}. At the same time, the physics literature shows that hyperviscosity can create bottleneck effects, distort the dissipation range, and partially thermalize the high-frequency spectrum \cite{HaugenBrandenburg2004,FrischEtAl2008,AgrawalAlexakisBrachetTuckerman2020}. This is precisely why a rigorous mathematical distinction between \emph{regularizing} and \emph{physically faithful} is important: the same term that helps regularity may also change the spectral cascade in a way that is undesirable for quantitative turbulence prediction.

\subsection{Spectral illustrations near the critical growth rate}

To visualize how the order $2\alpha$ changes the damping mechanism, we plot two linear spectral quantities. The first is the rate
\[
\lambda_\alpha(k)=\nu k^2 + \mu k^{2\alpha},
\]
which governs the decay of a Fourier mode of wavenumber $k$ in the linearized problem. The second is the decay factor
\[
E_\alpha(t;k_0)=\exp\bigl(-2(\nu k_0^2+\mu k_0^{2\alpha})t\bigr)
\]
for a fixed representative mode $k_0$.

These figures are not direct numerical simulations of the nonlinear PDE. They are analytic spectral illustrations meant to make the threshold mechanism visible. In particular, they show how the hyperdissipative component begins to dominate the Laplacian at high frequencies and how this dominance becomes stronger as $\alpha$ crosses the critical value $\frac54$.

\begin{figure}[ht]
\centering
\includegraphics[width=0.78\textwidth]{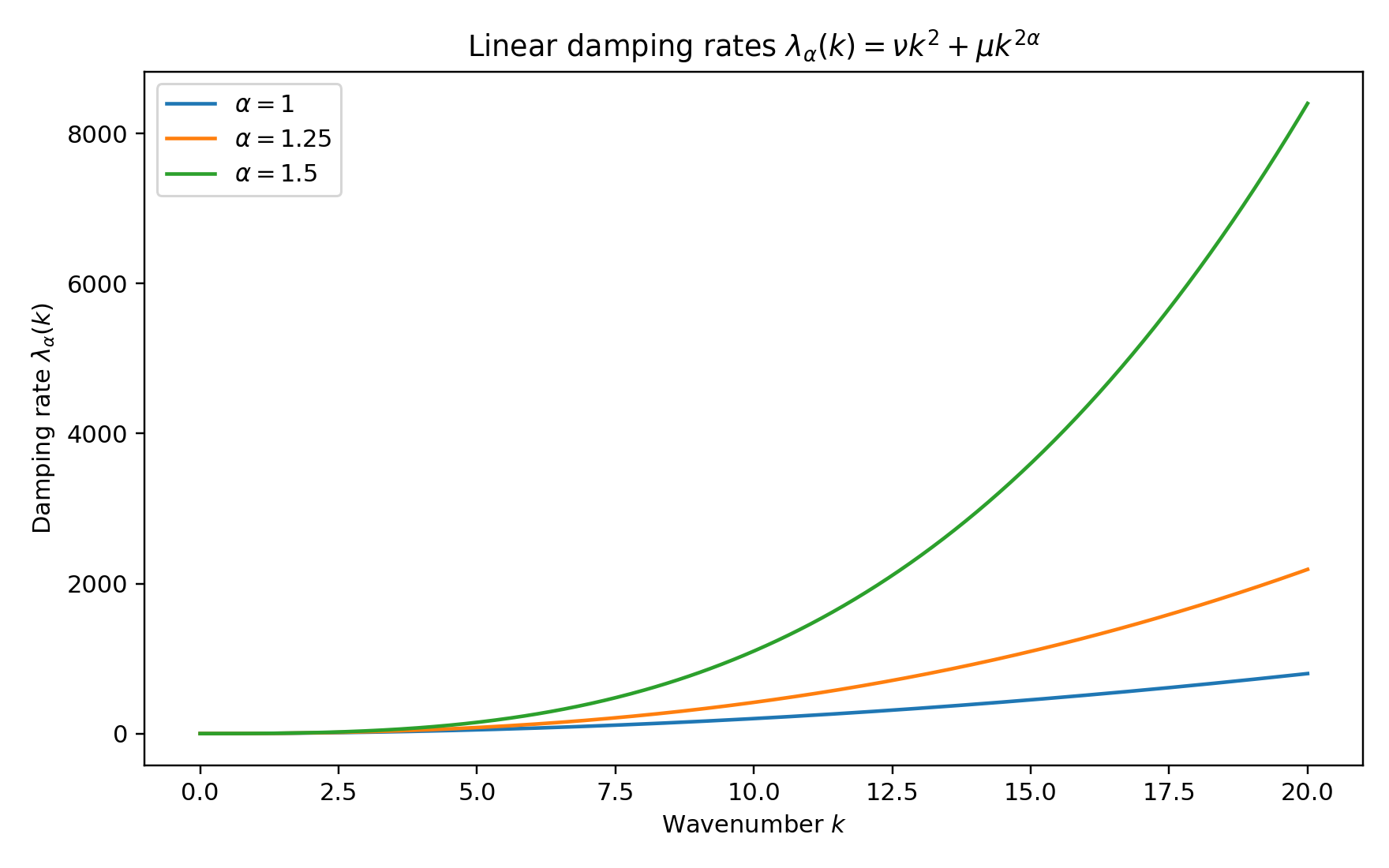}
\caption{Linear damping rate $\lambda_\alpha(k)=\nu k^2+\mu k^{2\alpha}$ for $\alpha=1$, $\alpha=\frac54$, and $\alpha=\frac32$. The separation between the curves increases rapidly at high wavenumbers, reflecting the growing strength of the nonlocal dissipation.}
\label{fig:damping}
\end{figure}

\begin{figure}[ht]
\centering
\includegraphics[width=0.78\textwidth]{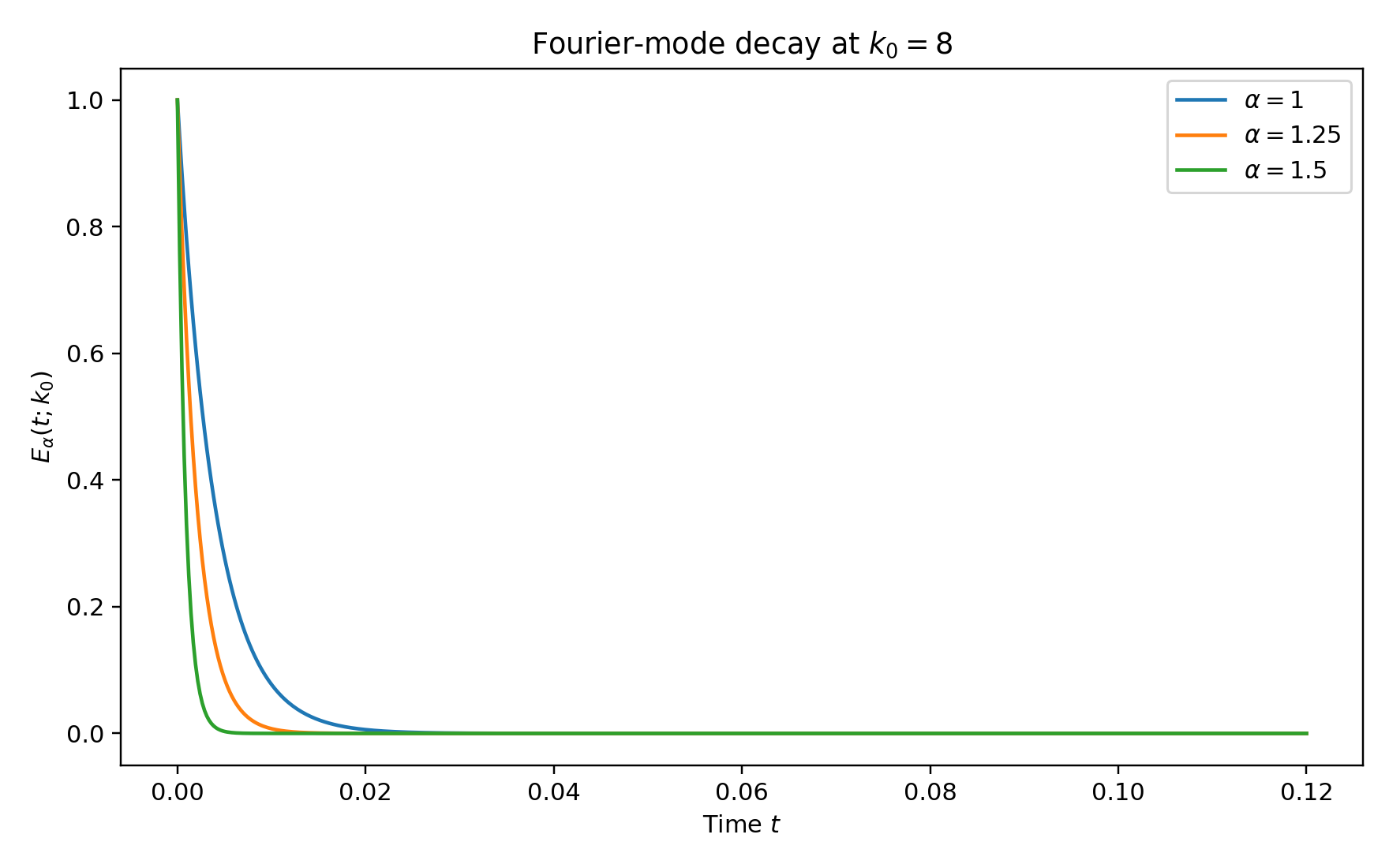}
\caption{Decay of a representative Fourier mode with $k_0=8$ under the linearized evolution. Larger $\alpha$ produces noticeably faster high-frequency damping.}
\label{fig:decay}
\end{figure}

\subsection{What remains open below the Lions threshold}

The interval
\[
1<\alpha<\frac54
\]
is still supercritical with respect to the finite-energy scaling, as shown in \cref{prop:scaling}. Our results give global weak solutions, local strong well-posedness, and a global small-data theorem in this range, but they do not settle the large-data global regularity question. This is exactly where the modified equation remains closest in spirit to the Millennium problem. Recent work of Gruji\'c and Xu \cite{GrujicXu2024,GrujicXu2025} emphasizes the continuing significance of this supercritical range, while the non-uniqueness results of Li, Qu, Zeng, and Zhang \cite{LiQuZengZhang2024} show that even in highly dissipative regimes one must distinguish very carefully between physically relevant solution classes and scaling-supercritical ones.

For future numerical work, the present paper suggests a natural agenda:
\begin{enumerate}[label=(\arabic*)]
\item compare Fourier-Galerkin or pseudospectral computations for $\alpha=1$, $\alpha=\frac54$, and $\alpha>\frac54$ under the same initial data and forcing;
\item examine how the bottleneck region and dissipation spectrum vary with the growth exponent $2\alpha$;
\item test nonlocal kernels $c_j$ whose symbols interpolate between bounded, first-order, Laplacian-order, and hyperdissipative behavior;
\item investigate whether effective variable-viscosity closures extracted from data produce symbols closer to the harmless or genuinely regularizing side of \cref{prop:criterion}.
\end{enumerate}
This numerical program is especially attractive because it can be carried out without changing the incompressibility structure or the spatial domain of the classical problem.

\section{Concluding remarks}

We have studied a class of nonlocal perturbations of the three-dimensional incompressible Navier--Stokes equations in which the only modification of the classical system is a selfadjoint dissipative Fourier multiplier. The analysis isolates the Fourier-symbol mechanism that matters for regularity: bounded or first-order convolution corrections are analytically lower order, whereas symbols comparable to $|\xi|^{2\alpha}$ with $\alpha>1$ produce a genuinely hyperdissipative regime.

The principal mathematical conclusions are as follows.
\begin{enumerate}[label=(\roman*)]
\item The exact $L^2$ energy identity persists for the full nonlocal class considered here.
\item Global weak solutions exist for arbitrary finite-energy divergence-free data whenever $\alpha>1$.
\item Local strong well-posedness holds in $H^s(\R^3)$ for every $s>\frac52$.
\item The critical threshold for arbitrary-data global strong solvability remains the Lions exponent $\alpha=\frac54$.
\item In the vanishing-hyperdissipation limit, the fixed-parameter global theory degenerates in a rigid way near a hypothetical singular time of the classical flow.
\end{enumerate}

From the point of view of the large-data three-dimensional problem, the last item is the structural conclusion of the paper. The hyperdissipative approximation provides a globally regular family for each fixed parameter, but the estimates responsible for that regularity are not uniform as the parameter tends to zero. Thus the approximation does not resolve the classical Navier--Stokes problem; rather, it identifies the precise analytic locus at which uniform control fails.

The class studied here also covers natural nonlocal viscosity laws of convolution type and includes the standard fractional Laplacian models as special cases. In that sense, the paper provides a common Hilbert-space framework for three issues that are often discussed separately: symbol-level classification of nonlocal dissipation, global well-posedness in the hyperdissipative regime, and the loss of uniformity in the Navier--Stokes limit.

Two natural directions remain open: sharpening the vanishing-parameter analysis in scale-critical continuation classes, and strengthening the crossover-frequency concentration obtained here beyond the one-sided estimates proved in this paper.

\section*{Conflict of interest}
The authors declare that they have no conflict of interest.

\section*{Data availability}
No datasets were generated or analyzed during the current study.

\section*{Author contributions}
Both authors contributed to the conception of the work, the mathematical analysis, the writing of the manuscript, and the approval of the final version.

\end{document}